\documentclass[10pt,reqno]{amsart}
\usepackage{amsfonts}
\usepackage{amsmath}
\usepackage{CJK}
\usepackage{amssymb}
\usepackage{latexsym,bm}
\usepackage[RGB]{xcolor}
\usepackage{mathrsfs}
\usepackage{amsthm}
\numberwithin{equation}{section}
\newtheorem{theorem}{Theorem}[section]

\newtheorem{lemma}[theorem]{Lemma}

\newtheorem{proposition}[theorem]{Proposition}
\newcommand{\dif}{\mathrm{d}}

\newcommand{\SUM}[3]{\sum\limits_{{#1}={#2}}^{#3}}

\newcommand{\LIM}[2]{\lim\limits_{{#1}\rightarrow{#2}}}

\begin{document}
\title[Navier-Stokes equations with Maxwell's law]{Asymptotic stability of rarefaction waves for compressible Navier-Stokes equations with relaxation}
\author{Yuxi Hu and Xuefang Wang}
\begin{abstract}
The asymptotic stability of rarefaction wave for 1-d relaxed  compressible  isentropic Navier-Stokes equations  is established. For initial data with different far-field values, we show that  there exists a unique global in time solution. Moreover,  as time goes to infinity, the obtained solutions are shown to  converge uniformly to  rarefaction wave solution of $p$-system  with corresponding Riemann initial data. The proof is based on $L^2$ energy methods.
 \\[2em]
 {\bf AMS classification code}: 35B25, 76N10
\end{abstract}
\maketitle
\section{ Introduction}
In this paper, we consider the system of one-dimensional isentropic compressible Navier Stokes equations in Lagrange coordinates as
\begin{align}\label{1.1}
\begin{cases}
 v_t=u_x,\\
u_t+ p(v)_x= \tilde S_x,
\end{cases}
\end{align}
with 
\begin{align}\label{1.2}
\tau  \tilde S_t+ \tilde S=\mu \frac{u_x}{v},
\end{align}
where $v, u, p, \tilde S$ denote the specific volume per unit mass, fluid velocity, pressure and stress tensor, respectively. The equations \eqref{1.1} are the 
consequences of conservation of mass and balance of momentum, respectively. The viscosity coefficient  $\mu$  is assumed to be a positive constant as well as the relaxation parameter $\tau$. Moreover, we assume the pressure $p$ satisfy $p(v)=a v^{-\gamma}$ with $\gamma>1$ being the adiabatic index.

In the constitutive relation \eqref{1.2}, $\tau$ is the relaxtion time describing the  time lag in the response of the stress tensor to the velocity gradient. In fact, even in simple fluid, water for example, the``time lag''  exists, but it is very small ranging from
1 ps to 1 ns, see \cite{MMA, SRK}. However, Pelton et al. \cite{PC} showed that such a``time lag'' cannot be neglected, even for simple fluids, in
the experiments of high-frequency (20GHZ) vibration of nano-scale mechanical devices immersed in water-glycerol mixtures. It turned out
  that, cp. also \cite{CS}, equation (\ref{1.2}) provides a general formalism  to characterize the fluid-structure interaction
of nano-scale mechanical devices vibrating in simple fluids. A similar relaxed constitutive relation was already proposed by Maxwell in \cite{MA}, in order to describe the relation of stress tensor and velocity gradient for a non-simple fluid.

The system \eqref{1.1} is coupled with the initial conditions
\begin{align} \label{1.3}
(v, u, \tilde S) (0,x)=(v_0, u_0, \tilde S_0)(x),\quad x\in \mathbb R,
\end{align}
with 
\[
\LIM x {\pm \infty}(v_0, u_0, \tilde S_0)(x)=(v_{\pm}, u_{\pm}, 0).
\]
Here, we assume, for the far field conditions, $v_+\neq v_-, u_+\neq u_-$ in general. Furthermore, $(v_-, u_-)$ and $(v_+, u_+)$ are supposed to be the Riemann initial data which generates a centered rarefaction wave
for the following $p-$system
\begin{align}\label{1.4}
\begin{cases}
v_t=u_x,\\
u_t+p(v)_x=0.
\end{cases}
\end{align}
That means, there exists a continuous weak solution with the form $(v^r, u^r)(x/t)$ of $p-$ system \eqref{1.4} with 
\begin{align}\label{1.5}
(v^r, u^r)(0, x)=\begin{cases}
(v_-, u_-), \quad x<0,\\
(v_+, u_+), \quad x>0
\end{cases}
\end{align}

The main purpose of this paper is devoted to establish the asymptotic stability of the rarefaction wave $(v^r, u^r, 0)$ defined above. More precisely, we assume the initial data $(v_0, u_0, \tilde S_0)$ is close
to $(v_0^r, u_0^r, 0)$ in a suitable sense and the amplitude $\delta=|v_+-v_-|+|u_+-u_-|$ of the rarefaction wave is sufficiently small. Then, we will show that there exists a unique global defined solution
$(v, u, \tilde S)$ to problem \eqref{1.1}-\eqref{1.3} which approaches the rarefaction wave $(v^r, u^r, 0)$ uniformly in $x$ as $t\rightarrow \infty$. That is, we have
\[
\|(v-v^r, u-u^r, \tilde S\|_{L^\infty} \rightarrow 0, \quad \mathrm{as}\quad t\rightarrow \infty.
\]
See Theorem \ref{th1.1} for more details.

Note that, when $\tau=0$,  the system  \eqref{1.1}, \eqref{1.2} reduce to the classical isentropic compressible Navier-Stokes equations, which read as
\begin{align}\label{1.6}
\begin{cases}
v_t=u_x,\\
u_t+p(v)_x=\left(\frac{\mu u_x}{v}\right)_x.
\end{cases}
\end{align}
For such system, the well-posedness and asymptotic stability results have been widely studied with various initial data, see \cite{KA, KZ, MN1, MN2}. In particular, Matsumura and Nishihara \cite{MN1} first established the asymptotic stability of rarefaction waves with small amplitude in 1986. Later, they \cite{MN2} extended their results to large data cases. The half-space problem was conducted by Kawashima and Zhu \cite{KZ}. A similar asymptotic stability result was also obtained for the full system of the compressible Navier-Stokes equations, see \cite{KMN}. 

The asymptotic stability of rarefaction waves were also shown for other related system, see \cite{MAT} for the Broadwell model, \cite{KT} for a model system of radiating gas, \cite{NNK, BHZ} for a model of hyperbolic balance law.

One should note that it is not obvious that the results which hold for the classical systems also hold for the relaxed system. Indeed, and for example, Hu and Wang \cite{HW}  and Hu, Racke and Wang \cite{HRW} showed that, for the one-dimensional isentropic and/or non-isentropic compressible Navier-Stokes system,  solutions  exist globally with arbitrary large initial data for classical system, while solutions blow up in finite time with some large initial data for the corresponding relaxed system. A similar qualitative change was observed before for certain thermoelastic systems, where the non-relaxed system is exponentially stable, while the relaxed one is not, see Quintanilla and Racke resp. Fern\'andez Sare and Mu\~noz Rivera \cite{QuRa011,FeMu012} for plates, and Fern\'andez Sare and Racke \cite{FeRa009} for Timoshenko beams.

The paper is organized as follows.  Some preliminaries and main theorem (Theorem \ref{th1.1}) are given in Section 2.  In Section 3, following to \cite{MN1, MN2} , we construct smooth approximations of rarefaction waves. The original problem is reformulated in Section 4 and there the statements of asymptotic stability results for the reformulated problems are given (Theorem \ref{th4.1}). The a priori estimates are also given in Section 4 (Proposition \ref{pro4.2}). Last, the proof of the a priori estimates (Proposition \ref{pro4.2}) are given in Section 5.

\textbf{Notations:}  $L^p(\mathbb R)$ and $W^{s,p}(\mathbb R)$  ($1\le p \le\infty$) denote the  usual Lebesgue  and Sobolev spaces over $\mathbb R$ with the norm $\|\cdot \|_{L^p}$ and $\|\cdot\|_{W^{s,p}}$, respectively. Note that, when $s=0$, $W^{0,p}=L^p$. For $p=2$, $W^{s, 2}$ are abbreviated to $H^s$ as usual. 
Let $T$ and $B$ be a positive constant and a Banach space, respectively. $C^k(0,T; B)(k \ge 0 )$ denotes the space of $B$-valued $k$-times continuously differentiable functions on $[0,T]$, and $L^p(0,T; B)$ denotes the space of $B$-valued $L^p$-functions on $[0,T]$. The corresponding space $B$-valued functions on  $[0,\infty)$ are defined similarly.

\section{ Preliminaries and main results}
Let's first recall the definition of rarefaction waves of p-system \eqref{1.4} with Riemann initial data \eqref{1.5}. For fixed $(v_-, u_-) (v_->0, u_-\in \mathbb R)$, we define a proper neighbourhood $\omega \subset \mathbb R_{v,u}^2:= \mathbb R_+\times  \mathbb R$ as 
\begin{align*}
R_1(v_-, u_-):= \left\{ (v, u)\in \omega | u=u_-- \int_{v_-}^v \lambda_1(s)\dif s, u\ge u_-\right\},\\
R_2(v_-, u_-):= \left\{ (v, u) \in \omega | u=u_--\int_{v_-}^v \lambda_2(s)\dif s, u\ge u_- \right\}
\end{align*}
and
\begin{align*}
RR (v_-, u_-):= \left\{ (v, u) \in \omega | u\ge u_--\int_{v_-}^v \lambda_1(s) \dif s, u\ge u_--\int_{v_-}^v \lambda_2(s) \dif s \right\},
\end{align*}
where $\lambda_1(v)=-\sqrt{-p^\prime(v)}$, $\lambda_2(v)=\sqrt{- p^\prime(v)}$ are eigenvalues of the matrix
\begin{align*}
\begin{pmatrix}
0& -1\\
p^\prime(v) & 0
\end{pmatrix}.
\end{align*}
If $(v_+, u_+) \in RR (v_-, u_-)$, then Riemann problem \eqref{1.4}, \eqref{1.5} admits a continuous weak solution of the form $(v^r, u^r)(x/t)$ which is called centered rarefaction wave.

Let 
\[
I_0:= \|(v_0- v_0^r, u_0-u_0^r)\|_{L^2}+\|((v_0)_x, (u_0)_x)\|_{H^1}+\|\tilde S_0\|_{H^2}.
\]
Our main result is stated as follows.
\begin{theorem}\label{th1.1}
Let $(v_+, u_+) \in RR (v_-, u_-)$ and $\delta=|v_+-v_-|+|u_+-u_-|$. Assume the initial data $(v_0, u_0, \tilde S_0)$ satisfy
\begin{align*}
(v_0- v_0^r, u_0-u_0^r, \tilde S_0) \in L^2,\\
( (v_0)_x, (u_0)_x)\in H^1, \tilde S_0 \in H^2
\end{align*}
and there exists a positive constant $\epsilon_0$ such that if $I_0+\delta < \epsilon_0$, then the initial value problem \eqref{1.1}-\eqref{1.3} has a unique global solution in time satisfying
\begin{align}\label{2.1}
\begin{cases}
(v-v_0^r, u-u_0^r, \tilde S)\in C^0(0, +\infty, L^2),\\
(v_x, u_x)\in C^0(0, +\infty, H^1)\cap L^2(0, +\infty, H^1), S\in C^0(0, +\infty, H^2)\cap L^2 (0, +\infty, H^2).
\end{cases}
\end{align}
Moreover, this solution approaches the rarefaction wave $(v^r, u^r, 0)$ uniformly in $x\in \mathbb R$ as $t\rightarrow +\infty$:
\begin{align}\label{2.2}
\|(v(t,x)-v^r(x/t), u(t,x)-u^r(x/t), \tilde S(t,x))\|_{L^\infty} \rightarrow 0, \quad \mathrm{as}\quad t\rightarrow \infty.
\end{align}
\end{theorem}

\section{Smooth approximation of rarefaction waves}
We start from the inviscid Burgers equation
\begin{align} \label{3.1}
\begin{cases}
w_t^r+w^r w_x^r=0,\\
w^r(0,x)=w_0^r(x)=
\begin{cases}
w_-, \quad x<0,\\
w_+,\quad x>0.
\end{cases}
\end{cases}
\end{align}
Let $w_-<w_+$, then \eqref{3.1} has a continuous weak solution $w^r(x/t)$ with
\begin{align}\label{3.2}
w^r(x/t)=
\begin{cases}
w_-,\quad x/t\le w_-,\\
x/t,\quad w_-\le x/t \le w_+,\\
w_+,\quad \xi \ge w_+.
\end{cases}
\end{align}
This solution is called central rarefaction wave connecting the two constant state $w_-$ and $w_+$. Next, we consider the following system whose solutions approximate $w^r(x/t)$ smoothly.
\begin{align} \label{3.3}
\begin{cases}
w_t+w w_x=0,\\
w(0,x)=w_0(x)=\hat w+\tilde w k_q \int_0^{\epsilon x} (1+y^2)^{-q} \dif y,
\end{cases}
\end{align}
where $\hat w= \frac{w_++w_-}{2}, \tilde w=\frac{w_+-w_-}{2}$, $\epsilon>0$ is a constant. $k_q$ is a constant satisfying $k_q \int_0^\infty (1+y^2)^{-q} \dif y=1$ for $q>\frac{3}{2}$. Then we have the following lemma
\begin{lemma} \label{le3.1}\cite{MN1, MN2}
Let $w_-< w_+$, then there exists a smooth solution $w(t,x)$ of system \eqref{3.3} satisfying \\

1) $w_-< w(t,x)<w_+$,  $w_x(t,x)>0,\quad \forall (t,x)\in \mathbb R_+ \times \mathbb R$\\

2) $\forall 1\le p \le \infty$, there exists a constant $C_p$ such that
\begin{align*}
\|w_x\|_{L^p}\le C_p \min \left( \epsilon^{1-\frac{1}{p}} \tilde w, \tilde w^\frac{1}{p} t^{-1+\frac{1}{p}}\right)\\
\|\partial_x^k w\|_{L^p} \le C \min \left( \epsilon^{2-\frac{1}{p}} \tilde w, \epsilon^{(1-\frac{1}{2q})(1-\frac{1}{p})} \tilde w ^{-\frac{p-1}{2pq}} t^{-1-\frac{p-1}{2pq}} \right).
\end{align*}

3) if $w_->0$, then 
\begin{align*}
|w(t,x)-w_-|\le C \tilde w (1+(\epsilon x)^2)^{-\frac{q}{3}}(1+(\epsilon w_- t)^2)^{-\frac{q}{3}},\\
|w_x|\le C \epsilon \tilde w (1+ (\epsilon x)^2)^{-\frac{q}{2}} (1+(\epsilon w_- t)^2)^{-\frac{q}{2}}.
\end{align*}

4) if $w_+<0$, then 
\begin{align*}
|w(t,x)-w_-|\le C \tilde w (1+(\epsilon x)^2)^{-\frac{q}{3}}(1+(\epsilon w_+ t)^2)^{-\frac{q}{3}},\\
|w_x|\le C \epsilon \tilde w (1+ (\epsilon x)^2)^{-\frac{q}{2}} (1+(\epsilon w_+ t)^2)^{-\frac{q}{2}}.
\end{align*}

5) $\LIM t \infty \sup_{\mathbb R} |w(t,x)-w^r(x/t)| =0$.
\end{lemma}
Next, we construct smooth solutions which approximate the weak solutions $(v^r, u^r)(x/t)$ of  system \eqref{1.4}-\eqref{1.5} by using $w(t,x)$.  Firstly, suppose $(v_+, u_+)\in R R (v_-, u_-)$, then 
there exists a unique pair $(\bar v, \bar u)$ satisfying
\[
(\bar v, \bar u) \in R_1(v_-, u_-),\quad  (v_+, u_+)\in R_2(\bar v, \bar u).
\]
Moreover, the continuous weak solution of system \eqref{1.4}-\eqref{1.5} can be expressed as 
\begin{align}\label{3.4}
(v^r, u^r)(\xi)= (v_1^r+v_2^r-\bar v, u_1^r+u_2^r-\bar u)(\xi),\quad, \xi:=x/t,
\end{align}
where
\begin{align*}
\lambda_1(v_1^r(\xi))=w_1^r(\xi), \lambda_2(v_2^r(\xi))=w_2^r(\xi),\\
u_1^r(\xi)=u_--\int_{v_1}^{v_1^r(\xi)} \lambda_1(s)\dif s, u_2^r(\xi)=\bar u- \int_{\bar v}^ {v_2^r(\xi)}\lambda_2(s)\dif s,
\end{align*}
and $w_1^r(\xi), w_2^r(\xi)$ are given by \eqref{3.2} with 
\begin{align*}
w_{1-}=\lambda_1(v_-), w_{1+}=\lambda_1(\bar v), w_{2-}=\lambda_2(\bar v), w_{2+}=\lambda_2(v_+).
\end{align*}
Now, we define $(V, U)$ as
\begin{align}\label{3.5}
\begin{cases}
(V, U)(t,x)=(V_1+V_2-\bar v, U_1+U_2-\bar u)(t,x),\\
\lambda_1(V_1)=w_1(t,x), \lambda_2(V_2)=w_2(t,x),\\
U_1=u_1-\int_{v_1}^{V_1} \lambda_1(s)\dif s, U_2=\bar u-\int_{\bar v}^{V_2} \lambda_2(s)\dif s
\end{cases}
\end{align}
where $w_1(t,x), w_2(t,x)$ are solutions of \eqref{3.3}. Note that $(V_1, U_1)$ and $(V_2, U_2)$ defined above are exact solution of $p$-system $\eqref{1.4}$ and $(V, U)$ satisfies
\begin{align}\label{3.6}
\begin{cases}
V_t-U_x=0,\\
U_t+p(V)_x=g(V)_{x},
\end{cases}
\end{align}
where $g(V)=p(V)-p(V_1)-p(V_2)+p (\bar v)$. 

By use of Lemma \ref{le3.1}, we have
\begin{lemma}\label{le3.2} \cite{MN1, MN2, NNK}
(V, U) defined in \eqref{3.5} satisfy

1) $V_t>0$,  $\forall (x,t)\in \mathbb R_+\times \mathbb R$,

2) There exists a constant $C$ such that 
\begin{align}
|V_x|\le C V_t,\quad V_t \le C \epsilon \delta,\quad \forall (t,x) \in \mathbb R_+\times \mathbb R
\end{align}

3) Denote $\tilde w_i=\frac{w_{i+}-w_{i-}}{2}$, then $\forall t\in \mathbb R_+$, 
\begin{align*}
\|g(V)_x\|_{L^p}\le C \epsilon^{1-\frac{1}{p}} \tilde w_1 \tilde w_2 \left\{ (1+(\epsilon w_{2-}t)^2)^{-\frac{q}{3}}+(1+(\epsilon w_{1+}t)^2)^{-\frac{q}{3}}\right\}
\end{align*}
and
\[
\int_0^\infty \|g(V)_x\|_{L^p} \dif x \le C \delta^2 \epsilon^{-\frac{1}{p}}.
\]

4)$\forall t\in \mathbb R_+ $, 
\[
\|V_x\|_{L^p}, \|U_x\|_{L^p} \le C \min \{ \delta \epsilon^{1-\frac{1}{p}}, \delta^\frac{1}{p} (1+t)^{-1+\frac{1}{p}}\}, 
\]

5) $\forall t\in \mathbb R_+ $ and $k=2, 3, 4$, 
\begin{align*}
\|\partial_x^k V\|_{L^p}, \|\partial_x^k U\|_{L^p} \le C \left( \delta^{-\frac{p-1}{2pq}} \epsilon^{(1-\frac{1}{2q})(1-\frac{1}{p})} (1+t)^{-1-\frac{p-1}{2pq}}+\delta^\frac{1}{p} (1+t)^{-2+\frac{1}{p}} \right)
\end{align*}
and $\forall p >1$, 
\[
\int_0^\infty ( \|\partial_x^k V\|_{L^p}+\|\partial_x^k U\|_{L^p}) \dif t \le C \delta^{-\frac{p-1}{2pq}}.
\]

6) $\LIM t \infty \sup_{\mathbb R} |(V, U)(t, x)- (v^r, u^r)(x/t)|=0$.
\end{lemma}

\section{Reformulation of the problem}
Let $(v^r, u^r)(x/t)$ be centered rarefaction wave for system \eqref{1.4}-\eqref{1.5} which is given in \eqref{3.4}, and let $(V, U)(t,x)$ be the smooth approximations of $(v^r, u^r)(x/t)$. Note that $(V, U)$ are the functions defined in \eqref{3.5} and satisfying \eqref{3.6}. Now, let $\varphi=v-V, \psi=u-U, S= \tilde S- \mu \frac{U_x}{V}$ be perturbations, then 
\begin{align}\label{4.1}
\begin{cases}
 \varphi_t= \psi_x,\\
 \psi_t+[ p(\varphi+V)-p(V)]_x- S_x= \mu  \left( \frac{U_x}{V} \right)_x-g(V)_x,\\
\tau S_t+ S-\frac{\mu}{v} \psi_x=-\tau \mu \left(\frac{ U_x}{V}\right)_t-\frac{ U_x \cdot \varphi}{v V},
\end{cases}
\end{align}
with
\begin{align} \label{4.2}
(\varphi, \psi, S)(x,0)=(\varphi_0, \psi_0, S_0),
\end{align}
where
\begin{align} \label{4.3}
\varphi_0=v_0-V_0, \psi_0=u_0-U_0, S_0= \tilde S_0- \mu \frac{(U_0)_x}{V_0}.
\end{align}
For this reformulated problem \eqref{4.1}-\eqref{4.2}, we have the following theorem of global existence and asymptotic stabiltiy. 
\begin{theorem}\label{th4.1}
Let $(v_+, u_+) \in RR (v_-, u_-)$ and $\delta=|v_+-v_-|+|u_+-u_-|$. Assume the initial data $(\varphi_0, \psi_0,  S_0)\in H^2$ and let $E_0:=\|(\varphi_0, \psi_0, S_0)\|_{H^2}$. Then, there exists a positive constant $\delta_1$ such tht if $E_0+\delta < \delta_1$, the initial value problem \eqref{4.1}-\eqref{4.2} has a unique global solution in time satisfying
\begin{align}
(\varphi, \psi, S) (t,x) \in C^0(0, +\infty, H^2), \label{4.8}\\
(\varphi_x, \psi_x)(t,x) \in L^2 (0, +\infty, H^1),\quad S(t,x) \in L^2(0, +\infty, H^2). \label{4.9}
\end{align}
Moreover, this solution decay to $(0,0,0)$ uniformly in $x$ as $t\rightarrow +\infty$:
\begin{align}\label{4.4}
\|(\varphi, \psi, S)\|_{W^{1, \infty}} \rightarrow 0, \quad \mathrm{as}\quad t\rightarrow +\infty.
\end{align}
\end{theorem}

To prove Theorem \ref{th4.1}, the key point is to show the a priori estimate of solutions to the problem  \eqref{4.1}, \eqref{4.2}.  To do this, we introduce the energy norm $E(t)$ as follows
$$E(t)= \sup_{0\le s \le t} \|(\varphi, \psi, S)(s, \cdot)\|_{H^2}+\int_0^t ( \| (\varphi_x, \psi_x, S_x)\|_{H^1}+\|\sqrt{V_t} \varphi\|^2+\|S\|^2 )\dif t. $$

The a priori estimate result is stated as follows.

\begin{proposition}\label{pro4.2}
Let $T>0$ and $(\varphi, \psi, S)(t,x)$ be a solution to the problem \eqref{4.1}, \eqref{4.2} such that 
\[
(\varphi, \psi, S)(t,x) \in C^0([0,T], H^2) \cap C^1([0, T], H^1).
\]
Then, there exists a positive constant $\delta_2$ which is independent of $T$ such that if 
\begin{align}\label{4.6}
E(T)+\delta< \delta_2, 
\end{align}
then the solution $(\varphi, \psi, S)$ satisfies
\begin{align}\label{4.5}
E(t)\le C(E_0+E^\frac{3}{2}(t)+\delta^\theta)
\end{align}
for $t\in (0, T)$ with $\theta=\min\{\frac{1}{2}, \frac{3}{2}-\frac{1}{q}\}$.
\end{proposition}
We will give the proof of Proposition \ref{pro4.2} in Section 5. 

{\it Proof of Theorem \ref{th4.1}}. Firstly, we choose $\delta_2$ such that $C \delta_2^\frac{1}{2} \le \frac{1}{2}$.  Then, under the assumption \eqref{4.6} and using \eqref{4.5}, we have
\begin{align}\label{4.7}
E(t)\le 2 C (E_0+\delta^\theta).
\end{align}
Now,  we choose $\epsilon_0$ small enough such that $ 2C(\epsilon_0+\epsilon_0^\theta)+\epsilon_0 < \frac{\delta_2}{2}$. Then, we get $E(t)+\delta<\frac{\delta_2}{2}$ which closes the assumption 
\eqref{4.6} by noting that $E_0+\delta<\epsilon_0<\frac{\delta_2}{2}$. Therefore, based on the local existence theorem and the a priori estimate \eqref{4.7}, we can get a global solution in time with regularity \eqref{4.8} and \eqref{4.9} by the classical continuation methods. 

Now, we show the convergence result \eqref{4.4}.  Firstly, by Sobolev interpolation theorem, we have
\begin{align*}
\|(\varphi, \psi, S)\|_{L^\infty} \le C \|(\varphi, \psi, S)\|_{L^2}^\frac{1}{2} \|(\varphi_x, \psi_x, S_x)\|_{L^2}^\frac{1}{2},\\
\|(\varphi_x, \psi_x, S_x)\|_{L^\infty} \le C \|(\varphi_x, \psi_x, S_x)\|_{L^2}^\frac{1}{2} \|(\varphi_{xx}, \psi_{xx}, S_{xx})\|_{L^2}^\frac{1}{2},
\end{align*}
Thus, the convergence $\|(\varphi_x, \psi_x, S_x)\|_{L^2} \rightarrow 0$ as $t\rightarrow +\infty$  implies \eqref{4.4} immediately.  Let $\Psi(t):= \|(\varphi_x, \psi_x, S_x)(t, \cdot)\|_{L^2}^2$. Then, by \eqref{4.9}, $\Psi(t)\in L^1(0,+\infty)$. From the system \eqref{4.1}, we can easily get  $\Psi^\prime(t)\in L^1(0,+\infty)$. Therefore, we get $\Psi(t) \in W^{1,1}(0, +\infty)$ which implies $\Psi(t)\rightarrow 0$ as $t\rightarrow +\infty$.  Thus the proof of Theorem \ref{th4.1} is complete.

Finally, by using Theorem \ref{th4.1}, we are able to show Theorem \ref{th1.1} hold. 

{\it Proof of Theorem 2.1.} Assume that $I_0+\delta$ is suitable small where
\[
I_0=\|(v_0-v_0^r, u_0-u_0^r)\|_{L^2}+\|((v_0)_x, (u_0)_x)\|_{H^1}+ \|\tilde S_0\|_{H^2},
\]
then we have
\begin{align*}
\|(\varphi_0, \psi_0)\|_{L^2}\le \|(v_0-v_0^r, u_0-u_0^r)\|_{L^2}+\|(v_0^r-V_0, u_0^r- U_0)\|_{L^2}\le I_0+C \delta,\\
\|\partial_x( \varphi_0, \psi_0)\|_{H^1}+\|S_0\|_{H^2}\le \|\partial_x(v_0, u_0)\|_{H^1}+\|\tilde S_0\|_{H^2}+ \|\partial_x(V_0, U_0)\|_{H^1}+ \|\mu \frac{\partial_xU_0}{V_0}\|_{H^2}\le I_0+C \delta.
\end{align*}
Therefore, we derive that $E_0+\delta\le I_0+C\delta$ which can be sufficiently small.  Therefore, by using Theorem \ref{4.1}, we get a unique global solution $(\varphi, \psi, S)$ to the problem \eqref{4.1}, \eqref{4.2}. Then, the functions defined by $(v, u, \tilde S)=(\varphi+V, \psi+U, S+\mu \frac{U_x}{V})$ solves the original problem \eqref{1.1}-\eqref{1.3}.  To show the convergence \eqref{2.2} hold, we note that
\[
\|(v-v^r, u-u^r, \tilde S)\|_{L^\infty} \le \|(\varphi, \psi, S)\|_{L^\infty}+\|(V-v^r, U-u^r, \mu \frac{U_x}{V})\|_{L^\infty} \rightarrow 0
\]
as $t\rightarrow +\infty$, where we used the convergence result \eqref{4.4} and Lemma \ref{le3.2}. This completes the proof of Theorem \ref{th1.1}.

\section{Proof of Proposition \ref{pro4.2}}
In the following lemmas, we always assume
that $E(t)+\delta \le \delta_2$ for some small $\delta_2$ and $\epsilon=\delta^3$. In particular, we have $0<V_-\le V\le V_+$ and $0<v_-\le v\le v_+$. Firstly, we get the following $L^2$ estimates.
\begin{lemma} \label{le5.1}
There exists a constant $C$ such that for $0\le t\le T$,
\begin{align}\label{5.1}
\|\varphi\|_{L^2}^2+\|\psi\|_{L^2}^2+\| S\|_{L^2}^2 +\int_0^t ( \|\sqrt{V_t} \varphi\|_{L^2}^2+\|S\|_{L^2}^2)\dif t \le C(E_0+\delta^\theta),
\end{align}
where $\theta=\min \{ \frac{1}{2}, \frac{3}{2}-\frac{1}{q} \}.$
\end{lemma}
\begin{proof}
Multiplying the equations $\eqref{4.1}_1, \eqref{4.1}_2, \eqref{4.1}_3$ by $p(V)-p(\varphi+V)$, $\psi$ and $\frac{v}{\mu} S$, respectively, and integrating over $\mathbb R$ with respect to $x$, we get
\begin{align}\label{5.2}
\frac{\dif}{\dif t} E_1(\varphi, \psi, S) +E_2(\varphi, S)=\int_{\mathbb R} \left( \mu \psi \left(\frac{U_x}{V}\right)_x-\psi g(V)_x+\tau v S \left( \frac{U_x}{V}\right)_t +\frac{\tau v_t}{2\mu} S^2\right) \dif x,
\end{align}
where 
\[
E_1(\varphi, \psi, S)=\int_{\mathbb R} \left( p(V)\varphi- \int_V^{V+\varphi} p(\xi)\dif \xi +\frac{1}{2} \psi^2 +\frac{\tau V}{2\mu} S^2 \right)\dif x
\]
and
\[
E_2(\varphi, S)=\int_{\mathbb R} \left(\frac{v}{\mu} S^2-\frac{V_t \varphi S}{V}+(p(V+\varphi)-p(V)-p^\prime (V) \varphi ) V_t \right) \dif x.
\]
Note that 
\[
p(V)\varphi-\int_V^{V+\varphi} p(\xi)\dif \xi \ge C \varphi^2,
\]
since $p(\cdot)$ is a convex function on $(0, +\infty)$. Therefore, we derive that
\begin{align}\label{5.3}
E_1(\varphi, \psi, S) \ge C \|(\varphi, \psi, S)\|_{L^2}^2.
\end{align}
On the other hand, let $p(V+\varphi)-p(V)-p^\prime (V) \varphi= f(v, V) \varphi^2$, then $f(v, V)\ge C >0$. Since $V_t\le C \epsilon \delta$ where $\epsilon$ and $\delta$ are sufficiently small, we have
\begin{align}\label{5.4}
 \frac{v}{\mu} S^2 -\frac{V_t \varphi S}{V} +(p(V+\varphi)-p(V)-p^\prime(V) \varphi) V_t \ge C( S^2+ V_t \varphi^2).
\end{align}
Now, we estimate the right-hand-side of \eqref{5.2}. Firstly, 
using Lemma \ref{3.2}, we get
\begin{align*}
&\int_{\mathbb R} \left| \mu \psi \left( \frac{U_x}{V} \right)_x \right| \dif x\le C \int_{\mathbb R} \left(|\psi U_{xx}|+|\psi U_x V_x|\right) \dif x\\
&\le C  \|\psi\|_{L^2} \left(\|U_{xx}\|_{L^2}+ \|U_x\|_{L^4} \|V_x\|_{L^4}\right)\le C \left( \delta^{\frac{3}{2}-\frac{1}{q}}(1+t)^{-1-\frac{1}{4q}}+\delta^\frac{1}{2} (1+t)^{-\frac{3}{2}} \right) \| \psi\|_{L^2}
\end{align*}
and
\begin{align*}
\int_{\mathbb R} |\psi g(V)_x| \dif x \le \|\psi\|_{L^2} \| g(V)_x\|_{L^2}.
\end{align*}
Moreover, using \eqref{3.6} and Lemma \ref{3.2}, we have
\begin{align*}
\int_{\mathbb R} \left| \tau v S \left(\frac{U_x}{V}\right)_t \right| \dif x &\le C \int_{\mathbb R} S(|g(V)_{xx}|+|p(V)_{xx}|+U_x^2) \dif x \\
&\le \varepsilon \int_{\mathbb R} S^2 \dif x +C(\varepsilon)( \|V_x\|_{L^4}^4+\|V_{xx}\|_{L^2}^2+\|U_x\|_{L^4}^4 ) \\
&\le C \varepsilon \int_{\mathbb R} S^2 \dif x +C \delta (1+t)^{-3}+( \delta^{\frac{3}{2}-\frac{1}{q}}(1+t)^{-1-\frac{1}{4q}}+\delta^\frac{1}{2} (1+t)^{-\frac{3}{2}})^2.
\end{align*}
Besides, note that $\|v_t\|_{L^\infty}=\|V_t+\varphi_t\|_{L^\infty} \le C (\epsilon \delta+\delta_1)$, we derive that
\begin{align*}
\int_{\mathbb R} \frac{\tau v_t}{2\mu} S^2 \dif x \le C (\delta^4+\delta_2) \int_{\mathbb R} S^2 \dif x,
\end{align*}
which can be absorbed by $E_2(\varphi, S)$ for suitable small $\delta$ and $\delta_2$.
Combining the above estimates,  and integrating \eqref{5.2} over $(0, t)$, we derive that
\begin{align*}
&\|\varphi\|_{L^2}^2+ \|\psi\|_{L^2}^2+\|S\|_{L^2}^2 +\int_0^t (\|\sqrt{V_t} \varphi\|_{L^2}^2+\|S\|_{L^2}^2 )\dif s\\
&\le CE_0+C \int_0^t ( \delta^\frac{1}{2} (1+t)^{-\frac{3}{2}}+\delta^{\frac{3}{2}-\frac{1}{q}}(1+t)^{-1-\frac{1}{4q}}+\|g(V)_x\|_{L^2}) \|\psi\|_{L^2} \dif s+ C \delta+ C \delta^{3-\frac{2}{q}}\\
& \le C(E_0+\delta^\theta),
\end{align*}
where we have used the fact that $\int_0^{+\infty} \|g(V)_x\|_{L^2} \dif t \le C \delta^2 \epsilon^{-\frac{1}{2}}=C \delta^\frac{1}{2}$ from Lemma \ref{le3.2}. Therefore, the proof of this lemma is completed.
\end{proof}

\begin{lemma}\label{le5.2}
There exists a constant $C$ such that for $0\le t\le T$, we have
\begin{align} \label{5.5}
\| \varphi_x\|_{H^1}^2+ \|\psi_x\|_{H^1}^2+ \| S_x\|_{H^1}^2 +\int_0^t \|S_x\|_{H^1}^2 \dif t \le C(E_0+\delta^\theta+E^\frac{3}{2}(t)).
\end{align}
\end{lemma}

\begin{proof}
By applying $\partial_x^k (k=1, 2)$ to the equations \eqref{4.1}, we derive the following system
\begin{align}\label{5.6}
\begin{cases}
\partial_t \partial_x^k \varphi =\partial_x^{k+1} \psi,\\
\partial_t \partial_x^k \psi+\partial_x^k(p^\prime(\varphi+V)\varphi_x)+\partial_x^k( (p^\prime(\varphi+V)-p^\prime(V))V_x)-\partial_x^{k+1} S=\mu \partial_x^{k+1} \left(\frac{U_x}{V}\right)-\partial_x^{k+1}g(V),\\
\tau \partial_t \partial_x^k S+\partial_x^k S-\partial_x^k \left(\frac{\mu}{v} \psi_x\right)=-\tau \mu \partial_t \partial_x^k \left(\frac{ U_x}{V}\right)-\partial_x^k\left(\frac{U_x \varphi}{vV}\right).
\end{cases}
\end{align}
Multiplying the above equations by $-p^\prime(\varphi+V) \partial_x^k \varphi$, $\partial_x^k \psi$, $\frac{v}{\mu}\partial_x^kS$, 
respectively, and integrating over $\mathbb R$, we get
\begin{align*}
\frac{\dif}{\dif t} \int_{\mathbb R} \left( -\frac{1}{2} p^\prime(\varphi+V) (\partial_x^k \varphi)^2+\frac{1}{2}(\partial_x^k\psi)^2+\frac{\tau v}{2\mu} (\partial_x^k S)^2 \right)\dif x
+\int_{\mathbb R} \left( \frac{v}{\mu}-\frac{\tau}{2\mu} \partial_t v\right) (\partial_x^k S)^2 \dif x=: \SUM j 1 8 R_k^j,
\end{align*}
where
\[
R_k^1=-\int_{\mathbb R} \frac{1}{2} p^{\prime \prime}(\varphi+V)( (\varphi_t + V_t) (\partial_x^k\varphi)^2+( \varphi_x+ V_x) \partial_x^k \psi \partial_x^k \varphi ) \dif x,\]
\[
R_k^2=- \int_{\mathbb R} \left( \partial_x^k (p^\prime(\varphi+V) \varphi_x)-p^\prime(\varphi+V) \partial_x^{k+1} \varphi\right) \partial_x^k \psi \dif x,
\]
\[
R_k^3=-\int_{\mathbb R} \partial_x^k (( p^\prime(\varphi+V)-p^\prime(V))V_x) \partial_x^k \psi \dif x, \quad R_k^4=\int_{\mathbb R} \mu \partial_x^{k+1} \left( \frac{ U_x}{V}\right)
\partial_x^k \psi \dif xV_x,
\]
\[
R_k^5=-\int_{\mathbb R} \partial_x^{k+1} g(V) \partial_x^k \psi \dif x, R_k^6=-\int_{\mathbb R} \tau \mu \partial_t \partial_x^{k}\left(\frac{ U_x}{V}\right) \partial_x^k S \dif x,
R_k^7=\int_{\mathbb R} \partial_x^k \left( \frac{U_x \varphi}{v V}\right) \partial_x^k S \dif x,
\]
\[
R_k^8=\int_{\mathbb R} ( \partial_x^k (\frac{\mu}{v} \psi_x)-\frac{\mu}{v} \partial_x^{k+1} \psi) \frac{v}{\mu} \partial_x^k S \dif x.
\]
We shall show that  for $1\le j \le 8, k=1, 2$,
\begin{align}\label{5.7}
\int_0^t R_k^j(t) \dif t \le C(\delta^\theta+E^\frac{3}{2}(t))+\varepsilon \int_0^t \| S_x\|_{H^1}^2 \dif t+C \int_0^t (1+t)^{-d} (\|\varphi_x\|_{H^1}^2+\| \psi_x\|_{H^1}^2) \dif t.
\end{align}
Firstly, 
\begin{align*}
\int_0^t R_k^1 \dif t \le C \int_0^t\int_{\mathbb R}( |\psi_x|+|\varphi_x|+|V_x|)((\partial_x^k \varphi)^2+(\partial_x^k \psi)^2 )\dif x \dif t \le C(E^\frac{3}{2}(t)+\delta^4).
\end{align*}
Secondly, we have for $k=1$ that
\begin{align*}
\int_0^t R_1^2 (t) \dif t =\int_0^t\int_{\mathbb R} p^{\prime \prime}(\varphi +V)(\varphi_x+V_x) \varphi_x \psi_x \dif x \dif t
\le C (E^\frac{3}{2}(t)+\delta^4).
\end{align*}
For $k=2$,
\begin{align*}
&\int_0^t R_2^2(t)\dif t\\
&=\int_0^t \int_{\mathbb R} \left( p^{\prime\prime\prime}(\varphi+V)(\varphi_x+V_x)^2 \varphi_x+p^{\prime\prime}(\varphi+V)( 3 \varphi_{xx}\varphi_x  +V_{xx} \varphi_x+2 V_x \varphi_{xx} )\right)\psi_{xx} \dif x \dif t\\
&\le C\int_0^t \|(\varphi_x, V_x, V_{xx})\|_{L^\infty} ( \|\varphi_x\|_{H^1}^2+\|\psi_x\|_{H^1}^2) \dif t\\
&\le C( E^\frac{3}{2}(t)+\delta^4+\delta^{3-\frac{2}{q}})+ C \int_0^t (1+t)^{-2} ( \|\varphi_x\|_{H^1}^2+\|\psi_x\|_{H^1}^2) \dif t.
\end{align*}
Similarly, for $k=1$, we have
\begin{align*}
&\int_0^t R_1^3 (t)\dif t\\
&=\int_0^t \int_{\mathbb R} \left(  (p^{\prime\prime}(\varphi+V)(\varphi_x+V_x)-p^{\prime\prime} (V)V_x) V_x+(p^\prime(\varphi+V)-p^\prime(V))  V_{xx}        \right) \psi_x \dif x \dif t\\
&\le C\int_0^t \left(  \|V_x\|_{L^\infty} ( \|\varphi_x\|_{H^1}^2+\|\psi_x\|_{H^1}^2 )+(\|V_x\|_{L^4}^2+\|V_{xx}\|_{L^2}) \|\partial_x\psi\|_{L^2} \right) \dif t\\
&\le C( \delta^4+ \delta^{\frac{3}{2}-\frac{1}{q}}+\delta^\frac{1}{2})
\end{align*}
and for $k=2$, 
\begin{align*}
&\int_0^t R_2^3 (t)\dif t \\
&\le C \int_0^t \left((E^\frac{1}{2}(t)+\|V_x\|_{L^\infty}+\|V_{xx}\|_{L^\infty}) ( \|\varphi_x\|_{H^1}^2+\|\psi_x\|_{H^1}^2 )+ C ( \|(V_x)^3 \|_{L^2}+\|V_{xx} V_x\|_{L^2}) \|\psi_{xx}\|_{L^2}  \right)\dif t\\
&\le C( E^\frac{3}{2}(t)+\delta^4+\delta^{3-\frac{2}{q}}+\delta^\frac{1}{2}+\delta^{\frac{11}{2}-\frac{1}{q}}).
\end{align*}
Now, we estimate $R_k^4$. For $k=1$, 
\begin{align*}
&\int_0^t R_1^4(t)\dif t=\int_0^t \int_{\mathbb R} \mu  \left( \frac{ U_x}{V} \right)_{xx} \psi_x \dif x\\
&=\int_0^t \int_{\mathbb R} \mu \left( \frac{U_{xxx}}{V}-\frac{2 U_{xx} V_x}{V^2}-\frac{U_x V_{xx}}{V^2}+\frac{2 V_x^2 U_x}{V^3}\right) \psi_x \dif x\\
&\le C \int_0^t ( \|U_{xxx}\|+\|U_{xx}\|_{L^4} \|V_x\|_{L^4}+\|V_{xx}\|_{L^4} \|U_x\|_{L^4}+ \|V_x^2\|_{L^4} \|U_x\|_{L^4}) \|\psi_x\| \dif t\\
&\le C\int_0^t (\delta^{\frac{3}{2}-\frac{1}{q}}(1+t)^{-1-\frac{1}{4q}}+\delta^\frac{1}{2} (1+t)^{-\frac{5}{2}}) \|\psi_x\| \dif t \le C (\delta^{\frac{3}{2}-\frac{1}{q}} +\delta^\frac{1}{2} ).
\end{align*}
For $k=2$, some tedious calculations give
\begin{align*}
&\partial_x^3 \left(\frac{U_x}{V}\right)\\
&= \frac{\partial_x^4 U}{V}-\frac{ 3 \partial_x^3 U V_x+3 U_{xx} V_{xx}+U_x \partial_x^3 V}{V^2}
+\frac{6 U_{xx} (V_x)^2+6 U_x  V_{xx} V_x}{V^3}-\frac{6 U_x (V_x)^3}{V^4}.
\end{align*}
So, we have
\begin{align*}
&\int_0^t R_2^4 (t)\dif t =\int_0^t \int_{\mathbb R} \mu \partial_x^3 \left(\frac{U_x}{V}\right) \psi_{xx} \dif x\\
&\le C \int_0^t \left( \|\partial_x^4 U\|_{L^2}+\|\partial_x^3 U V_x\|_{L^2}+ \|U_{xx}V_{xx}\|_{L^2}+ \|U_x \partial_x^3 V\|_{L^2}+\|U_x V_{xx} 
V_x\|_{L^2}+\|U_x (V_x)^3\|_{L^2}\right) \|\psi_{xx}\|_{L^2} \dif t\\
&\le C \int_0^t ( \delta^{\frac{3}{2}-\frac{1}{q}}(1+t)^{-1-\frac{1}{4q}}+\delta^\frac{1}{2} (1+t)^{-\frac{3}{2}}+\delta^{\frac{9}{2}-\frac{3}{q}}(1+t)^{-2-\frac{3}{4q}}
+\delta^\frac{9}{2} (1+t)^{-\frac{5}{2}}) \|\psi_{xx}\|_{L^2} \dif t \\
&\le C( \delta^{\frac{3}{2}-\frac{1}{q}}+\delta^\frac{1}{2}).
\end{align*}
Similarly, for $k=1, 2$, we can get
\begin{align*}
\int_0^t R_k^5 \dif t& \le C \int_0^t ( \|((V_x)^2, (V_x)^3)\|_{L^2}+\|(V_{xx}, \partial_x^3 V)\|_{L^2}+\|V_x V_{xx}\|_{L^2}) \|(\partial_x, \partial_x^2) \psi\|_{L^2} \dif t\\
&\le C \int_0^t ( \delta^{\frac{3}{2}-\frac{1}{q}} (1+t)^{-1-\frac{1}{4q}}+\delta^\frac{1}{2} (1+t)^{-\frac{3}{2}}) \|(\partial_x, \partial_x^2)\psi\|_{L^2} \dif t\le C(\delta^{\frac{3}{2}-\frac{1}{q}}+\delta^\frac{1}{2}).
\end{align*}
On the other hand, $\int_0^t R_k^6 (t)\dif t$ and $\int_0^t R_k^7(t)\dif t$ are estimated as follows. Note that for $k=1$, we have
\begin{align*}
\left( \frac{U_x}{V}\right)_{tx} =\frac{ (U_{xxt} V+U_{xt} V_x-U_{xx}V_t-U_xV_{tx})V^2-(U_{xt}V-U_x V_t)2V V_x}{V^4},\\
\left(\frac{U_x \varphi}{V v} \right)_x=\frac{ (U_{xx} \varphi+U_x \varphi_x)V v-U_x \varphi (V_x v+V v_x)}{ V^2 v^2}.
\end{align*}
So, we derive that
\begin{align*}
&\int_0^t R_1^6 (t)\dif t \le \varepsilon \int_0^t \int_{\mathbb R} \frac{v}{\mu} S_x^2 \dif x \dif t +C(\varepsilon) \int_0^t \int_{\mathbb R} \left| \left(\frac{U_x}{V}\right)_{tx} ^2 \right| \dif x \dif t\\
&\le  \varepsilon \int_0^t \int_{\mathbb R} \frac{v}{\mu} S_x^2 \dif x \dif t +C(\varepsilon)  \int_0^t \int_{\mathbb R}| V_{xxx}+V_{xx}V_x+V_{x}^3+U_{xx} U_x+U_x^2V_x |^2 \dif x \dif t\\
&\le \varepsilon  \int_0^t \int_{\mathbb R} \frac{v}{\mu} S_x^2 \dif x \dif t+C(\varepsilon) \int_0^t (\|V_x\|_{L^6}^6+\|V_x\|_{L^4}^2 (\|V_{xx}\|_{L^4}^2+\|U_{xx}\|_{L^4}^2+\|U_x\|_{L^8}^4)+\|V_{xxx}\|_{L^2}^2)\dif t\\
&\le \varepsilon  \int_0^t \int_{\mathbb R}\frac{v}{\mu} S_x^2 \dif x\dif t+\int_0^t ( \delta (1+t)^{-3}+\delta^{5-\frac{3}{q}}(1+t)^{-\frac{7}{2}-\frac{3}{4q}}+\delta^{3-\frac{2}{q}}(1+t)^{-2-\frac{1}{2q}} )\dif t \\
&\le \varepsilon  \int_0^t \int_{\mathbb R}\frac{v}{\mu} S_x^2 \dif x\dif t+\delta+\delta^{3-\frac{2}{q}} .
\end{align*}
and
\begin{align*}
&\int_0^t R_1^7 (t) \dif t \le \varepsilon   \int_0^t \int_{\mathbb R}\frac{v}{\mu} S_x^2 \dif x \dif t+C(\varepsilon) \int_0^t \int_{\mathbb R}  | U_{xx}\varphi+U_x \varphi_x+U_x V_x \varphi+U_x \varphi \varphi_x|^2 \dif x\dif t \\
&\le  \varepsilon  \int_0^t \int_{\mathbb R}\frac{v}{\mu} S_x^2 \dif x \dif t+C(\varepsilon)((E^\frac{3}{2}(t)+\delta^\frac{1}{2}+\delta^{3-\frac{2}{q}})+ \int_0^t (1+t)^{-2}\|\varphi_x\|_{L^2}^2\dif t).
\end{align*}
The estimates for $\int_0^t R_2^6(t)\dif t $ and $\int_0^t R_2^7(t)\dif t$ can be done in a similar way, we omit the details. 
Finally, let's estimate the term$\int_0^t R_k^8 (t)\dif t$. Firstly, for $k=1$, we have
\begin{align*}
\int_0^t R_1^8(t)\dif t &=-\int_0^t \int_{\mathbb R} \frac{v_x \psi_x  S_x}{v} \dif x \dif t\\
&\le \varepsilon  \int_0^t \int_{\mathbb R}\frac{v}{\mu} S_x^2 \dif x\dif t+C(\varepsilon) \int_0^t \int_{\mathbb R} (V_x^2+\varphi_x^2)\psi_x^2 \dif x \dif t\\
&\le \varepsilon  \int_0^t \int_{\mathbb R}\frac{v}{\mu} S_x^2 \dif x\dif t+ C(\varepsilon) (E^2(t)+\delta^8).
\end{align*}
and for $k=2$,
\begin{align*}
\int_0^t R_2^8(t) \dif t &=\int_0^t \int_{\mathbb R} \frac{2 v_x^2 \partial_x\psi-2 v v_x \psi_{xx}}{v^2} S_{xx}\dif x \dif t\\
&\le \varepsilon  \int_0^t \int_{\mathbb R}\frac{v}{\mu} S_{xx}^2+C(\varepsilon)\int_0^t \int_{\mathbb R} ( (|V_x|^4+|\varphi_x|^4)|\psi_x|^2+(|V_x|^2+|\varphi_x|^2) |\psi_{xx}|^2 )\dif x \dif t\\
&\le \varepsilon  \int_0^t \int_{\mathbb R}\frac{v}{\mu} S_{xx}^2 \dif x\dif t+C(\varepsilon) (E^2(t)+\delta^8).
\end{align*}

 Combining the above estimates and using the Gronwall's inequality, we get the desired results.
\end{proof}

The next lemma gives the dissipative estimates of $\varphi_x$ and $\psi_x$. 

\begin{lemma}\label{le5.3}
There exists a constant $C$ such that for $0\le t\le T$, we have
\begin{align}\label{5.8}
\int_0^t (\|\varphi_x\|_{H^1}^2+\|\psi_x\|_{H^1}^2) \dif t \le C(E_0+\delta^\theta+E^\frac{3}{2}(t)).
\end{align}
\end{lemma}

\begin{proof}
Multiplying the equation $\eqref{5.6}_2$ by $\partial_x^{k+1} \varphi$ for $k=0$ or $1$,  and integrating over $(0, t)\times \mathbb R$, we get
\begin{align*}
\int_0^t \int_{\mathbb R} - p^\prime(\varphi+V) (\partial_x^{k+1} \varphi)^2 \dif x \dif t=: \SUM j 1 6 M_k^j,
\end{align*}
where
\[
M_k^1=\int_0^t \int_{\mathbb R} \partial_t \partial_x^k \psi \cdot \partial_x^{k+1} \varphi \dif x \dif t,
M_k^2=\int_0^t \int_{\mathbb R} (\partial_x^k(p^\prime(\varphi+V) \varphi_x)-p^\prime(\varphi+V) \partial_x^{k+1}\varphi) \cdot \partial_x^{k+1} \varphi \dif x \dif t,
\]
\[
M_k^3=\int_0^t \int_{\mathbb R} \partial_x^k ( (p^\prime(\varphi+V)-p^\prime(V)) V_x) \cdot \partial_x^{k+1}\varphi \dif x \dif t, M_k^4=\int_0^t \int_{\mathbb R} \partial_x^{k+1} S \cdot \partial_x^{k+1} \varphi \dif x \dif t,
\]
\[
M_k^5=-\int_0^t \int_{\mathbb R} \mu \partial_x^{k+1} \left(\frac{U_x}{V}\right) \partial_x^{k+1}\varphi \dif x \dif t, M_k^6=\int_0^t \int_{\mathbb R} \partial_x^{k+1} g(V) \partial_x^{k+1} \varphi \dif x \dif t.
\]
We shall  show that
\begin{align}\label{5.9}
\SUM j 1 6 M_k^j \le \varepsilon \int_0^t \int_{\mathbb R} (\partial_x^{k+1} \varphi)^2 \dif x \dif t +C\int_0^t \int_{\mathbb R} (\partial_x^{k+1}\psi)^2 \dif x \dif t+C(E^\frac{3}{2}(t)+E_0+\delta^\theta).
\end{align}
Firstly, by doing integration by part and using equation $\eqref{5.6}_1$, we get
\begin{align*}
&\int_0^t \int_{\mathbb R} \partial_t \partial_x^k \psi \partial_x^{k+1}\varphi \dif x \dif t\\
&=\int_{\mathbb R} \partial_x^k \psi(t) \partial_x^{k+1} \varphi(t) \dif x-\int_{\mathbb R} \partial_x^k \psi_0 \partial_x^{k+1} \varphi_0 \dif x
+\int_0^t \int_{\mathbb R}\partial_x^{k+1}\psi \partial_t \partial_x^k \varphi \dif x \dif t\\
&\le \int_0^t \int_{\mathbb R} (\partial_x^{k+1} \psi)^2 \dif x \dif t + C( E^\frac{3}{2}(t)+\delta^\theta+E_0).
\end{align*}
Secondly, for $k=0$, $M_0^2$ vanishes and for $k=1$,
\begin{align*}
M_1^2&=\int_0^t \int_{\mathbb R} p^{\prime\prime}(\varphi+V)(\varphi_x+V_x)\varphi_x \varphi_{xx} \dif x \dif t\\
&\le \varepsilon \int_0^t \int_{\mathbb R}(\varphi_{xx})^2 \dif x \dif t +C\int_0^t \int_{\mathbb R}(\varphi_x^4+V_x^2\varphi_x^2) \dif x\\
&\le \varepsilon  \int_0^t \int_{\mathbb R} (\varphi_{xx})^2\dif x\dif t+C(E^2(t)+\delta^8).
\end{align*}
Similarly, we get for $k=0$, 
\begin{align*}
M_0^3 &\le \varepsilon \int_0^t \int_{\mathbb R} (\varphi_x)^2 \dif x \dif t+ C\int_0^t \int_{\mathbb R} \varphi^2 V_x^2\dif x \dif t \\
&\le \varepsilon \int_0^t \int_{\mathbb R} (\varphi_x)^2 \dif x \dif t+\sup \|\varphi\|_{L^2}^2 \int_0^t \|V_x\|_{L^\infty}^2 \dif t\\
&\le \varepsilon \int_0^t \int_{\mathbb R} (\varphi_x)^2 \dif x \dif t +C(E_0+\delta),
\end{align*}
and $k=1$,
\begin{align*}
M_1^3&\le \varepsilon \int_0^t \int_{\mathbb R} (\varphi_{xx})^2 \dif x \dif t +\int_0^t \int_{\mathbb R}( V_{xx}^2+V_x^4+\varphi_x^2 V_x^2)\dif x \dif t \\
&\le \varepsilon \int_0^t \int_{\mathbb R}(\varphi_{xx})^2 \dif x \dif t
+C(\delta^{3-\frac{2}{q}}+\delta).
\end{align*}
By using the dissipation of $S_x$, we have
\begin{align*}
M_k^4&\le \varepsilon \int_0^t \int_{\mathbb R} (\partial_x^{k+1}\varphi)^2 \dif x\dif t+C \int_0^t \int_{\mathbb R} (\partial_x^{k+1} S)^2 \dif \dif t \\
&\le \varepsilon \int_0^t \int_{\mathbb R} (\partial_x^{k+1}\varphi)^2\dif x \dif t +C(E^\frac{3}{2}(t)+E_0+\delta^\theta).
\end{align*}
$M_k^5$ and $M_k^6$ can be estimated in the same way as in Lemme \ref{le5.2}. Indeed, we have
\begin{align*}
M_0^5&\le C \int_0^t \int_{\mathbb R}( |\partial_{xx} U| +|U_x| |V_x| ) |\varphi_x| \dif x \dif t\\
&\le \varepsilon \int_0^t \int_{\mathbb R} |\varphi_x|^2 \dif x \dif t+C\int_0^t ( \|\partial_{xx}U\|_{L^2}^2+\|U_x|_{L^4}^4+|V_x|_{L^4}^4 )\dif t\\
&\le  \varepsilon \int_0^t \int_{\mathbb R} |\varphi_x|^2 \dif x \dif t
+C(\delta^{3-\frac{2}{q}}+\delta),
\end{align*}
and
\begin{align*}
M_1^5&\le  C \int_0^t ( \|U_{xxx}\|+\|U_{xx}\|_{L^4} \|V_x\|_{L^4}+\|V_{xx}\|_{L^4} \|U_x\|_{L^4}+ \|V_x^2\|_{L^4} \|U_x\|_{L^4}) \|\partial_{xx}\varphi\| \dif t\\
&\le \varepsilon \int_0^t \int_{\mathbb R} (\partial_{xx} \varphi)^2 \dif x \dif t +C( \delta^{3-\frac{2}{q}}+\delta^{5-\frac{3}{q}}+\delta).
\end{align*}
Note that $\int_0^t \|\partial_x g(V)\|_{L^2} \dif t \le C\delta^\frac{1}{2}$, we get
\[
M_0^6\le \sup \|\varphi_x\|_{L^2} \int_0^t \|\partial_x g(V)\|_{L^2} \dif t \le C \delta^\frac{1}{2}.
\]
Besides, we have
\begin{align*}
M_1^6\le \int_0^t \int_{\mathbb R} ( |V_{xx}|+|V_x|^2) |\varphi_{xx}| \dif x \dif t \le \varepsilon \int_0^t \int_{\mathbb R} (\varphi_{xx})^2 \dif x \dif t+C(\delta^{3-\frac{2}{q}}+\delta).
\end{align*}
Combining the above estimates, summing up $k$ from $0$ to $1$, we finally get
\begin{align}\label{5.10}
\int_0^t \|\varphi_x\|_{H^1}^2 \dif t \le C \int_0^t \|\psi_x\|_{H^1}^2 \dif t+ C(E^\frac{3}{2}(t)+E_0+\delta^\theta).
\end{align}
Now, we derive the dissipative estimates of $\psi_x$. Multiplying the equation by $\partial_x^{k+1}\psi$, we derive that
\begin{align*}
\int_0^t \int_{\mathbb R} \frac{\mu}{v} (\partial_x^{k+1} \psi)^2 \dif x \dif t =: \SUM j 1 5 N_k^j,
\end{align*}
where
\[
N_k^1=\int_0^t \int_{\mathbb R}  \tau \partial_t \partial_x^k S \partial_x^{k+1}\psi \dif x \dif t,
N_k^2=\int_0^t \int_{\mathbb R}  \partial_x^k S \partial_x^{k+1}\psi \dif x \dif t,
\]
\[
N_k^3=-\int_0^t \int_{\mathbb R}  \left( \partial_x^k \left(\frac{\mu}{v} \partial_x\psi\right)-\frac{\mu}{v}\partial_x^{k+1}\psi\right) \partial_x^{k+1}\psi \dif x \dif t,
\]
\[
N_k^4=\int_0^t \int_{\mathbb R}  \tau \mu \partial_t \partial_x^k \left(\frac{U_x}{V}\right) \partial_x^{k+1} \psi \dif x \dif t,
N_k^5=\int_0^t \int_{\mathbb R}  \partial_x^k \left( \frac{U_x \varphi}{v V}\right) \partial_x^{k+1} \psi \dif x \dif t.
\]
we shall show that
\begin{align}\label{5.11}
\SUM j 1 5 N_k^j\le \frac{1}{2}\int_0^t \int_{\mathbb R} \frac{\mu}{v} (\partial_x^{k+1} \psi)^2 \dif x \dif t + \varepsilon \int_0^t \|\partial_x\varphi\|_{H^1}^2 \dif t + C(E^\frac{3}{2}+E_0+\delta^\theta),
\end{align}
which implies that
\begin{align}\label{5.12}
\int_0^t \|\psi_x\|_{H^1}^2 \dif t \le \varepsilon \int_0^t \|\varphi_x\|_{H^1}^2 \dif t+ C(E^\frac{3}{2}(t)+E_0+\delta^\theta).
\end{align}
Combining \eqref{5.10} and \eqref{5.12} together, we get the desired result.  Therefore, we only need to show \eqref{5.11} hold. 

Firstly, using Lemma \ref{le5.2} and integrating by part with respect to $x$ and $t$, we have
\begin{align*}
N_k^1\le C(E^\frac{3}{2}(t)+\delta^\theta+E_0)+\int_0^t \int_{\mathbb R}  \partial_x^{k+1} S\partial_t \partial_x^k \psi \dif x \dif t,
\end{align*}
while, using equation $\eqref{5.6}_2$ and Lemma \ref{le5.2}, and exploiting the same methods as above (estimates of $M_k^j, j=2, \dots 6$, for example),  we get
\begin{align*}
&\int_0^t \int_{\mathbb R}  \tau \partial_x^{k+1} S \partial_t \partial_x^k \psi \dif x \dif t\\
&=\int_0^t \int_{\mathbb R}  \tau \partial_x^{k+1} S \left( -\partial_x^k (p^\prime(\varphi+V)\varphi_x)-\partial_x^k( (p^\prime(\varphi+V)-p^\prime(V))V_x)\right.\\
&\left. \qquad\qquad\qquad \qquad\qquad\qquad\qquad+\tau \partial_x^{k+1}S+
\mu \partial_x^{k+1} \left(\frac{U_x}{V}\right)-\partial_x^{k+1}g(V)\right) \dif x \dif t\\
&\le \varepsilon \int_0^t\int_{\mathbb R} (\partial_x^{k+1} \varphi)^2 \dif x \dif t+C(\varepsilon) \int_0^t \int_{\mathbb R} (\partial_x^{k+1} S)^2 \dif x \dif t +C (E^\frac{3}{2}(t)+\delta^\theta+E_0)\\
&\le \varepsilon \int_0^t\int_{\mathbb R} (\partial_x^{k+1} \varphi)^2 \dif x \dif t+C (E^\frac{3}{2}(t)+\delta^\theta+E_0),
\end{align*}
which gives the estimates of $N_k^1$.  Moreover, it's easy to see that
\begin{align*}
N_k^2&\le \frac{1}{8} \int_0^t \int_{\mathbb R} \frac{\mu}{v} (\partial_x^{k+1} \psi)^2 \dif x \dif t+ C \int_0^t \| S\|_{H^1}^2 \dif t \\
&\le \frac{1}{8} \int_0^t \int_{\mathbb R} \frac{\mu}{v} (\partial_x^{k+1} \psi)^2 \dif x \dif t+C(E^\frac{3}{2}(t)+\delta^\theta+E_0).
\end{align*}
Now, we estimate the term $N_k^j, j=3, 4, 5$. Note that, for $k=0$, $N_0^3$ vanishes, and for $k=1$, 
\begin{align*}
N_1^3&=\int_0^t\int_{\mathbb R} \frac{\mu v_x}{v^2} \psi_x \psi_{xx} \dif x \dif t\\
&\le \frac{1}{8} \int_0^t \int_{\mathbb R} \frac{\mu}{v} (\psi_{xx})^2\dif x \dif t +C\int_0^t\int_{\mathbb R} (|V_x|^2+|\varphi_x|^2)  |\psi_x|^2 \dif x \dif t\\
&\le \frac{1}{8} \int_0^t \int_{\mathbb R} \frac{\mu}{v} (\psi_{xx})^2\dif x \dif t + C(E^2(t)+\delta^8).
\end{align*}
On the other hand, for $k=0$, using the equation \eqref{3.6} and Lemma \ref{le3.2}, we have
\begin{align*}
N_0^4+N_0^5&\le C \int_0^t\int_{\mathbb R} \left( |g(V)_{xx}|+|p(V)_{xx}|+|U_x|^2+|U_x|  |\varphi| \right) |\psi_x| \dif x \dif t \\
&\le \frac{1}{8} \int_0^t\int_{\mathbb R} \frac{\mu}{v} |\psi_x|^2 \dif x \dif t +C (E_0+\delta+\delta^{3-\frac{2}{q}}),
\end{align*}
and for $k=1$, comparing with the estimates of $R_1^6$ and $R_1^7$, we do have
\begin{align*}
N_1^4+N_1^5 \le \frac{1}{8} \int_0^t\int_{\mathbb R} \frac{\mu}{v} |\psi_{xx}|^2 \dif x \dif t +C(E^\frac{3}{2}(t)+E_0+\delta^\theta).
\end{align*}
This finish the proof of the Lemma.
\end{proof}
Combining the results of Lemma \ref{le5.1}-\ref{5.3}, the proof of the Proposition \ref{pro4.2} is finished.

\textbf{Acknowledgement:} Yuxi Hu's research is supported by NNSFC (Grant No. 11701556) and Yue Qi Young Scholar project, China University of Mining and Technology (Beijing).

\vspace*{3em}
 \noindent {\footnotesize  Yuxi Hu, Department of Mathematics, China University of Mining and Technology, Beijing, 100083, P.R. China, yxhu86@163.com\\[1em]
 Xuefang Wang,  Tianjin Binhai New Area Dagang No.8 Middle school, Tianjin, 300270, P.R. China.
}


\begin{thebibliography}{aaaaa}
\bibitem{BHZ} Y. Bai, L. He and H. Zhao, Nonlinear stability of rarefaction waves for a hyperbolic system with Cattaneo's law, {\it Comm. Pure. Appl. Anal.} {\bf 20} (2021), 2441-2474.

\bibitem{CS} D. Chakraborty and J.E. Sader, Constitutive models for linear compressible viscoelastic flows of simple liquids at nanometer length scales,
{\it Physics of Fluids} {\bf 27} (2015), 052002-1--052002-13.

\bibitem{FeMu012} H.D. Fern\'andez Sare and J.E. Mu\~noz Rivera, Optimal rates of decay in 2-d thermoelasticity with second sound, {\it J. Math. Phys.} {\bf 53} (2012), 073509.

\bibitem{FeRa009} H.D. Fern\'andez Sare and R. Racke, On the stability of damped Timoshenko systems -- Cattaneo versus Fourier law. {\it Arch. Rational Mech. Anal.} {\bf 194} (2009), 221-251.


\bibitem{HW} Y. Hu and N. Wang, Global existence versus blow-up results for one dimensional compressible Navier-Stokes equations with Maxwell's law, {\it Math. Nachr.} {\bf 292} (2019), 826-840.

\bibitem{HRW}  Y. Hu, R. Racke and N. Wang, Formation of singularities for one-dimensional relaxed compressible Navier-Stokes equations, {\it J. Diff. Eqs.} {\bf 327}  (2022), 145-165.




\bibitem{KA}{Y.I. Kanel, On a model system of equations of one-dimensional gas motions, {\it J.Differ.Equ.} {\bf 4} (1968), 374-380.}

\bibitem{KZ} S. Kawashima and P. Zhu, Asymptotic stability of rarefaction wave for the Navier-Stokes equations for a compressible fluid wave in the half space,{\it Arch. Rat. Mech. Anal.}, {\bf 194}  (2009), 105-132.

\bibitem{KMN} S. Kawashima, A. Matsumura and K. Nishihara, Asymptotic behavior of solutions for the equations of a viscous heat conductive gas, {\it Proc. Japan Acad.}, {\bf 62} (1986), 249-252.



\bibitem{KT} S. Kawashima and Y. Tanaka, Stability of rarefaction waves for a model system of a radiating gas, {\it Kyushu J, Math.}, {\bf 58} (2004), 211-250.
%\bibitem{KN}
%     A. Matsumura, K. Nishihara, On the stability of the traveling wave solutions of a one-dimensional model system for compressible viscous gas, Jpn. J. Appl. Math. 2 (1) (1985) 17-25.

\bibitem{MAT} A. Matsumura, Asymptotic toward rarefaction wave of solutions of the Broadwell model of a discrete velocity gas, {\it Japan J. Appl. Math.}, {\bf 4} (1987), 489-502.


\bibitem{MN1} A. Matsumura and K. Nishihara, Asymptotic toward the rarefaction waves of solutions of a one-dimensional model system for compressible viscous gas, {\it Japan. J. Appl. Math.}, {\bf 3} (1986), 1-13.

\bibitem{MN2} A. Matsumura and K. Nishihara, Global stability of the rarefaction waves of a one-dimensional model system for compressible viscous gas, {\it Comm. Math. Phys.}, {\bf 144} (1992), 325-335.

\bibitem{MMA} G. Maisano, P. Migliardo, F. Aliotta, C. Vasi, F. Wanderlingh and G. D'Arrigo, Evidence of anomalous acoustic behavior from brillouinscattering in supercooledvater, {\it Phys. Rev. Lett.} {\bf 52} (1984), 1025.


\bibitem{MA} J.C. Maxwell, On the dynamics theory of gases, {\it Phil. Trans. R. Soc. Lond.} {\bf 157} (1867), 49-88.


\bibitem{NNK} K. Nakamura, T. Nakamura and S. Kawashima, Asymptotic stability of rarefaction waves for a hyperbolic system of balance laws, {\it Kinetic and Related Models}, {\bf 12} (4), 2019, 923-944.

\bibitem{PC} M. Pelton, D. Chakraborty,E. Malachosky, P. Guyot-Sionnest and J. E. Sader, Viscoelastic flows in simple liquids generated by vibrating nanostructures,
{\it Phys. Rev. Lett.}, {\bf 111} (2013), 244502.

\bibitem{QuRa011} R. Quintanilla and R. Racke, Addendum to: Qualitative aspects of solutions in resonators, {\it Arch. Mech.} {\bf 63} (2011), 429-435.


\bibitem{SRK} F. Sette, G. Ruocco, M. Krisch, U. Bergmann, C. Masciovecchio, V. Mazzacurati, G. Signorelli and R. Verbeni, Collective dynamics in water by high energy resolution inelastic X-Ray scattering, {\it Phys. Rev. Lett.}, {\bf 75} (1995), 850.


%\bibitem{ChSa015} D. Chakraborty and J.E. Sader, Constitutive models for linear compressible viscoelastic flows of simple liquids at nanometer length scales, {\it Physics of Fluids} {\bf 27} (2015), 052002-1--052002-13.
%\bibitem{CJ}{Y. Cho and B.J. Jin, Blow-up of viscous heat-conducting compressible flows, {\it J. Math. Anal. Appl.} {\bf 320} (2) (2006), 819-826. }
%\bibitem{CK} H.J. Choe and H. Kim, Strong solutions of the Navier-Stokes equations for isentropic compressible
%fluids, {\it J. Differ. Eqs.} {\bf 190} (2003), 504-523.
%\bibitem{DA1} D. Hoff, Global existence for 1D, compressible, isentropic Navier-Stokes equations
%with large initial data, {\it Trans. Amer. Math. Soc.} {\bf 303} (1) (1987), 169-181.
%\bibitem{DA2} D. Hoff, Global solutions of the Navier-Stokes equations for multidimensional compressible
%flow with discontinuous initial data, {\it J. Differ. Eqs.} {\bf 120} (1) (1995), 215-254.
% \bibitem{FE} E. Feireisl, A. Novotny and H. Petzeltov\'{a}, On the existence of globally defined weak solutions to
%the Navier-Stokes equations, {\it J. Math. Fluid Mech.} {\bf 3} (2001), 358-392.
%\bibitem{HR} Y. Hu and R. Racke, Compressible Navier-Stokes equations with hyperbolic heat
%conduction, {\it J. Hyper. Diff. Equations} (2015), to appear.
%\bibitem{HLX} X.D. Huang, J. Li and Z.P. Xin, Global well-posedness of classical solutions with large
%oscillations and vacuum to the three-dimensional
%isentropic compressible Navier-Stokes equations, {\it Comm. Pure. Appl. Math.} {\bf 65} (2012), 549-585.
% \bibitem{JZ01} S. Jiang and P. Zhang, Global spherically symmetry solutions of the compressible isentropic
%Navier-Stokes equations, {\it Comm. Math. Phys.} {\bf 215} (2001), 559-581.
%\bibitem{JZ03} S. Jiang and P. Zhang, Axisymmetric solutions of the 3-D Navier-Stokes equations for compressible isentropic fluids, {\it J. Math. Pures. Appl.} {\bf 82} (2003), 949-973.
%\bibitem{KW} S. Kawashima, Systems of a hyperbolic-parabolic composite type, with
%applications to the equations of magnetohydrodynamics, Thesis, Kyoto
%University (1983).
% \bibitem{LI1} P.L. Lions, {\it Mathematical Topics in Fluid Mechanics}, Vol.I, Incompressible Models. Clarendon
%Press, Oxford (1996).
% \bibitem{LI2} P.L. Lions, {\it Mathematical Topics in Fluid Mechanics}, Vol.II, Compressible Models. Clarendon
%Press, Oxford (1998).
%\bibitem{Ma84} A. Majda {\it Compressible fluid flow and systems of
%  conservation laws in several space variables.} Appl.\ Math.\ Sci.\ {\bf 53}.
%  Springer, New York et al.\ (1984).
%\bibitem{MN} A. Matsumura and T. Nishida, The initial value problem for the
%equations of motion of viscous and heat-conductive gases, {\it J. Math.
%Kyoto Univ.} {\bf 20} (1) (1980), 67-104.
%\bibitem{MN1} A. Matsumura and T. Nishida, Initial Boundary Value Problems for the Equations of Motion
%of Compressible Viscous and Heat-Conductive Fluids, {\it Comm. Math. Phys.} {\bf 89} (1983), 445-464.
%\bibitem{NA} J. Nash, Le probl\`{e}me de Cauchy pour les \'{e}quations
%diff\'{e}rentielles d'un fluide g\'{e}n\'{e}ral, {\it Bull. Soc. Math.
%France} {\bf 90} (1962), 487-497.
%\bibitem{Ra015} R. Racke, {\it Lectures on Nonlinear Evolution Equations. Initial Value Problems,} 2nd edition, Birkh\"auser, Basel (2015).
%\bibitem{RS1} R. Racke and J. Saal, Hyperbolic Navier-Stokes equations I: Local well-posedness, {Evolution Equations and Control Theory} {\bf 1} (2012), 195-215.
%\bibitem{RS2} R. Racke and J. Saal. Hyperbolic Navier-Stokes equations II: Global existence of small solutions, {\it Evolution Equations and Control Theory} {\bf 1} (2012), 217-234.
%  \bibitem{SE} J. Serrin, On the uniqueness of compressible fluid motion, {\it Arch. Rational Mech. Anal.} {\bf 3} (1959), 271-288.
%  \bibitem{Sch012} A. Sch\"owe, A quasilinear delayed hyperbolic Navier-Stokes system: global solution, asymptotics and relaxation limit, {\it Meth. Appl. Anal.} {\bf 19} (2012), 99-118.
%      \bibitem{Sch014} A. Sch\"owe, Global strong solution for large data to the hyperbolic Navier-Stokes equation, {\it arxiv.org/abs/1409.7797} (2014).
%\bibitem{SK} Y. Shizuta and S. Kawashima, Systems of equations of
%   hyperbolic-parabolic type with applications to the discrete Boltzmann equation.
%   {\it Hokkaido Math. J.} {\bf 14} (1985), 249-275.
%\bibitem{X} Z.P. Xin, Blowup of smooth solutions to the compressible Navier-Stokes equation
%with compact density, {\it Comm. Pure. Appl. Math.} {\bf 51} (1998), 229-240.
% \bibitem{YO1}  W.-A. Yong, A note on the zero mach number limit of compressible euler
%equations, {\it Proc. Amer. Math. Soc.}, {\bf 133} (2005), 3079-3085.
%\bibitem{YO2} W.-A. Yong, Newtonian limit of Maxwell fluid flows, {\it Arch. Rational Mech. Anal.} {\bf 214} (2014) 913-922.
\end{thebibliography}
\end{document}